\documentclass[preprint,12pt]{elsarticle}

\usepackage{amssymb}
\usepackage{amsthm}
\usepackage{amsmath}
\usepackage{epic}
\usepackage{CJK}
\usepackage{setspace}
\newtheorem{thm}{Theorem}[section]
\newtheorem{cor}[thm]{Corollary}
\newtheorem{lem}[thm]{Lemma}

\journal{~}

\begin{document}
\begin{spacing}{1.15}
\begin{frontmatter}
\title{\textbf{Inverse Perron values and connectivity of a uniform hypergraph}}

\author[label1,label2]{Changjiang Bu}
\ead{buchangjiang@hrbeu.edu.cn}
\author[label1]{Haifeng Li}
\author[label2]{Jiang Zhou}


\address{
\address[label1]{College of Automation, Harbin Engineering University, Harbin 150001, PR China}
\address[label2]{College of Science, Harbin Engineering University, Harbin 150001, PR China}

}

\begin{abstract}
In this paper, we show that a uniform hypergraph $\mathcal{G}$ is connected if and only if one of its inverse Perron values is larger than $0$. We give some bounds on the bipartition width, isoperimetric number and eccentricities of $\mathcal{G}$ in terms of inverse Perron values. By using the inverse Perron values, we give an estimation of the edge connectivity of a $2$-design, and determine the explicit edge connectivity of a symmetric design. Moreover, relations between the inverse Perron values and resistance distance of a connected graph are presented.
\end{abstract}

\end{frontmatter}

\section{Introduction}
Let $V(\mathcal{G})$ and $E(\mathcal{G})$ denote the vertex set and edge set of a hypergraph $\mathcal{G}$, respectively. $\mathcal{G}$ is $k$-uniform if $|e|=k$ for each $e\in E(\mathcal{G})$. In particular, $2$-uniform hypergraphs are usual graphs. For $i\in V(\mathcal{G})$, $E_i(\mathcal{G})$ denotes the set of edges containing $i$, and $d_i=|E_i(\mathcal{G})|$ denotes the degree of $i$. The adjacency tensor \cite{9} of a $k$-uniform hypergraph $\mathcal{G}$, denoted by $\mathcal{A}_\mathcal{G}$, is an order $k$ dimension $|V(\mathcal{G})|$ tensor with entries
\[a_{i_1 i_2  \cdots i_k }  = \left\{ \begin{gathered}
  \frac{1}
{{\left( {k - 1} \right)!}},{\kern 1pt} {\kern 1pt} {\kern 1pt} {\kern 1pt} {\kern 1pt} {\kern 1pt} {\kern 1pt} {\kern 1pt} {\kern 1pt} {\kern 1pt} {\kern 1pt} {\kern 1pt} {\kern 1pt} {\kern 1pt} {\kern 1pt} {\kern 1pt} {\kern 1pt} {\kern 1pt} {\kern 1pt} {\kern 1pt} {\kern 1pt} {\kern 1pt} {\kern 1pt} {\kern 1pt} {\kern 1pt} {\kern 1pt} {\kern 1pt} {\kern 1pt} {\kern 1pt} if~\left\{ {\left. {i_1 ,i_2 , \ldots ,i_k } \right\}} \right. \in E\left( \mathcal{G} \right), \hfill \\
  0,{\kern 1pt} {\kern 1pt} {\kern 1pt} {\kern 1pt} {\kern 1pt} {\kern 1pt} {\kern 1pt} {\kern 1pt} {\kern 1pt} {\kern 1pt} {\kern 1pt} {\kern 1pt} {\kern 1pt} {\kern 1pt} {\kern 1pt} {\kern 1pt} {\kern 1pt} {\kern 1pt} {\kern 1pt} {\kern 1pt} {\kern 1pt} {\kern 1pt} {\kern 1pt} {\kern 1pt} {\kern 1pt} {\kern 1pt} {\kern 1pt} {\kern 1pt} {\kern 1pt} {\kern 1pt} {\kern 1pt} {\kern 1pt} {\kern 1pt} {\kern 1pt} {\kern 1pt} {\kern 1pt} {\kern 1pt} {\kern 1pt} {\kern 1pt} {\kern 1pt} {\kern 1pt} {\kern 1pt} {\kern 1pt} {\kern 1pt} {\kern 1pt} {\kern 1pt} {\kern 1pt} {\kern 1pt} {\kern 1pt} {\kern 1pt} {\kern 1pt} {\kern 1pt} {\kern 1pt} {\kern 1pt} {\kern 1pt} {\kern 1pt} {\kern 1pt} {\kern 1pt} {\kern 1pt} {\kern 1pt} {\kern 1pt} {\kern 1pt} {\kern 1pt} {\kern 1pt} {\kern 1pt} {\kern 1pt} {\kern 1pt} {{otherwise}}. \hfill \\
\end{gathered}  \right.\]
The \textit{Laplacian tensor} \cite{H+} of $\mathcal{G}$ is $\mathcal{L}_\mathcal{G}=\mathcal{D}_\mathcal{G}-\mathcal{A}_\mathcal{G}$, where $\mathcal{D}_\mathcal{G}$ is the diagonal tensor of vertex degrees of $\mathcal{G}$. Recently, the research on spectral hypergraph theory via tensors has attracted much attention [7-10,14,19,24]. The spectral properties of the Laplacian tensor of hypergraphs are studied in [13,25,27,29,35].

The algebraic connectivity of a graph plays important roles in spectral graph theory \cite{Fiedler}. Analogue to the algebraic connectivity of a graph, Qi \cite{H+} defined the \textit{analytic connectivity} of a $k$-uniform hypergraph $\mathcal{G}$ as
\[
\alpha(\mathcal{G})= \min\limits_{j=1,\ldots,n} \min \left\{\mathcal{L}_\mathcal{G}\mathbf{x}^k:\mathbf{x}\in \mathbb{R}^n_+, \sum_{i=1}^n x_i^k=1, x_j=0 \right\},
\]
where $n=|V(G)|$, $\mathbb{R}^n_+$ denotes the set of nonnegative vectors of dimension $n$. Qi proved that $\mathcal{G}$ is connected if and only if $\alpha(\mathcal{G})>0$. In \cite{liwei}, some bounds on $\alpha(\mathcal{G})$ were presented in terms of degree, vertex connectivity, diameter and isoperimetric number. A feasible trust region algorithm of $\alpha(\mathcal{G})$ was give in \cite{cuichunfeng}.

For any vertex $j$ of uniform hypergraph $\mathcal{G}$, we define the \textit{inverse Perron value} of $j$ as
\[
\alpha_j(\mathcal{G})= \min \left\{\mathcal{L}_\mathcal{G}\mathbf{x}^k:\mathbf{x}\in \mathbb{R}^n_+, \sum_{i=1}^n x_i^k=1, x_j=0 \right\}.
\]
Clearly, the analytic connectivity $\alpha(\mathcal{G})=\min\limits_{j\in V(\mathcal{G})}\alpha_j(\mathcal{G})$ is the minimum inverse Perron value. For a connected graph $G$, $\alpha_j(G)$ is the minimum eigenvalue of $\mathcal{L}_G(j)$, where $\mathcal{L}_G(j)$ is the principal submatrix of $\mathcal{L}_G$ obtained by deleting the row and column corresponding to $j$. $\mathcal{L}_G(j)$ is a nonsingular $M$-matrix, and its inverse $\mathcal{L}_G(j)^{-1}$ is a nonnegative matrix \cite{Kirkland}. It is easy to see that $\alpha_j^{-1}(G)$ is the spectral radius of $\mathcal{L}_G(j)^{-1}$, which is called the Perron value of $G$. The Perron values have close relations with the Fielder vector of a tree \cite{a,b}.

The resistance distance \cite{Klein93,zhou17} is a distance function on graphs. For two vertices $i,j$ in a connected graph $G$, the \textit{resistance distance} between $i$ and $j$, denoted by $r_{ij}(G)$, is defined to be the effective resistance between them when unit resistors are placed on every edge of $G$. The \textit{Kirchhoff index} \cite{Klein93,ZhouB} of $G$, denoted by $Kf(G)$, is defined as the sum of resistance distances between all pairs of vertices in $G$, i.e., $Kf(G)=\sum_{\{i,j\}\subseteq V(G)}r_{ij}(G)$. $Kf(G)$ is a global robustness index. The resistance distance and Kirchhoff index in graphs have been investigated extensively in mathematical and chemical literatures [3-6,12,26,35,40].

This paper is organized as follows. In Section 2, some auxiliary lemmas are introduced. In Section 3, we show that a uniform hypergraph $\mathcal{G}$ is connected if and only if one of its inverse Perron values is larger than $0$, and some inequalities among the inverse Perron values, bipartition width, isoperimetric number and eccentricities of $\mathcal{G}$ are established. Partial results improve some bounds in \cite{liwei,H+}. We also use the inverse Perron values to estimate the edge connectivity of $2$-designs. In Section 4, some inequalities among the inverse Perron values, resistance distance and Kirchhoff index of a connected graph are presented.
\section{Preliminaries}
For a positive integer $n$, let $[n]=\{1,2,\ldots,n\}$. An order $m$ dimension $n$ tensor $\mathcal{ T} = (t_{i_1  \cdots i_m } )$ consists of $n^m$ entries, where $i_j  \in [n],\,j\in [m]$. When $m=2$, ${T}$ is an $n\times n$ matrix. Let $\mathbb{R}^{[m,n]}$ denote the set of order $m$ dimension $n$ real tensors, and let $\mathbb{R}_+^{n}$ denote the $n$-dimensional nonnegative vector space. For $\mathcal{T}=(t_{i_1i_2\cdots i_m})\in \mathbb{R}^{[m,n]}$ and $\mathbf{x}=\left({x_1,\ldots,x_n }\right)^\mathrm{T}\in\mathbb{R}^n$, let ${\mathcal{T}\mathbf{x}^{m - 1} }\in \mathbb{R}^n$ denote the vector whose $i$-th component is
\[
\left( {\mathcal{T}\mathbf{x}^{m - 1} } \right)_i  = \sum\limits_{i_2 ,i_3 , \ldots ,i_m  = 1}^n {t_{ii_2  \cdots i_m } x_{i_2 } x_{i_3 }  \cdots x_{i_m } },
\]
and let $\mathcal{T}\mathbf{x}^m$ denote the following polynomial
\[
\mathcal{T}\mathbf{x}^m=\sum\limits_{i_1,\ldots,i_m=1}^nt_{i_1i_2\cdots i_m} x_{i_1}\cdots x_{i_m}.
\]
In 2005, Qi \cite{RA2005} and Lim \cite{Chang} proposed the concept of eigenvalues of tensors, independently. For $\mathcal{T} = \left( {t_{i_1 i_2  \cdots i_m } } \right) \in \mathbb{R}^{\left[ {m,n} \right]}$, if there exists a number $\lambda\in\mathbb{R}$ and a nonzero vector $\mathbf{x}=\left({x_1,\ldots,x_n }\right)^\mathrm{T}\in\mathbb{R}^n$ such that $\mathcal{T}\mathbf{x}^{m - 1}  = \lambda \mathbf{x}^{\left[ {m - 1} \right]}$, then $\lambda$ is called an \textit{H-eigenvalue} of $\mathcal{T}$, $\mathbf{x}$ is called an \textit{H-eigenvector} of $\mathcal{T}$ corresponding to $\lambda$, where $\mathbf{x}^{\left[ {m - 1} \right]}=(x_1^{m-1},\ldots,x_n^{m-1})^\mathrm{T}$.

For a vertex $j$ of a $k$-uniform hypergraph $\mathcal{G}$, let $\mathcal{L}_\mathcal{G}(j)\in \mathbb{R}^{[k,n-1]}$ denote the principal subtensor of $\mathcal{L_G}\in \mathbb{R}^{[k,n]}$ with index set $V(\mathcal{G})\setminus\{j\}$. By Lemma 2.3 in \cite{M-tensor}, we have the following lemma.
\begin{lem}\label{2.1}
Let $\mathcal{G}$ be a $k$-uniform hypergraph. For any $j\in V(\mathcal{G})$, $\alpha_j(\mathcal{G})$ is the smallest H-eigenvalue of $\mathcal{L}_\mathcal{G}(j)$.
\end{lem}

A path $\mathcal{P}$ of a uniform hypergraph $\mathcal{G}$ is  an alternating sequence of vertices and edges $v_{0}e_{1}v_{1}e_{2} \cdots v_{l-1}e_{l}v_{l}$, where $v_0,\ldots,v_{l}$ are distinct vertices of $\mathcal{G}$, $e_1,\ldots,e_l$ are distinct edges of $\mathcal{G}$ and $v_{i - 1} ,v_i  \in e_i $, for $i = 1, \ldots ,l$. The number of edges in $P$ is the length of $P$. For all $u,v \in V(\mathcal{G})$, if there exists a path starting at $u$ and terminating at $v$, then $\mathcal{G}$ is said to be \textit{connected} \cite{A. Bretto}.

\begin{lem}\rm\cite{H+}\label{connected}
The uniform hypergraph $\mathcal{G}$ is connected if and only if $\alpha(\mathcal{G})>0$.
\end{lem}

Let $\mathcal{G}$ be a $k$-uniform hypergraph, $S$ be a proper nonempty subset of $V(\mathcal{G})$. Denote $\overline{S}=V(\mathcal{G})\setminus S$. The edge-cut set $E(S,\overline{S})$ consists of edges whose vertices are in both $S$ and $\overline{S}$. The minimum cardinality of such an edge-cut set is called \textit{edge connectivity} of $\mathcal{G}$, denote by $e(\mathcal{G})$.

\begin{lem}\rm\cite{H+}\label{edge}
Let $\mathcal{G}$ be a $k$-uniform hypergraph with $n$ vertices. Then
\[
e({\mathcal{G}})\geq\frac{n}{k}\alpha(\mathcal{G}).\]
\end{lem}

The \textit{$\{1\}$-inverse} of a matrix $M$ is a matrix $X$ such that $MXM=M$. Let $M^{(1)}$ denote any $\{1\}$-inverse of $M$, and let $(M)_{ij}$ denote the $(i,j)$-entry of $M$.
\begin{lem}\rm\cite{Bapat04,zhou17}\label{2.4}
Let $G$ be a connected graph. Then
\begin{eqnarray*}
r_{ij}(G)=(\mathcal{L}_G^{(1)})_{ii}+(\mathcal{L}_G^{(1)})_{jj}-(\mathcal{L}_G^{(1)})_{ij}-(\mathcal{L}_G^{(1)})_{ji}.
\end{eqnarray*}
\end{lem}
Let $\mbox{\rm tr}(A)$ denote the trace of the square matrix $A$, and let $e$ denote an all-ones column vector.
\begin{lem}\rm\cite{Sun}\label{2.5}
Let $G$ be a connected graph of order $n$. Then
\begin{eqnarray*}
Kf(G)=n\mbox{\rm tr}(\mathcal{L}_G^{(1)})-e^\top\mathcal{L}_G^{(1)}e.
\end{eqnarray*}
\end{lem}

\begin{lem}\rm\cite{Zhou2}\label{2.6}
Let $G$ be a connected graph of order $n$. Then $\begin{pmatrix}\mathcal{L}_G(i)^{-1}&0\\0&0\end{pmatrix}\in\mathbb{R}^{n\times n}$ is a symmetric $\{1\}$-inverse of $\mathcal{L}_G$, where $i$ is the vertex corresponding to the last row of $\mathcal{L}_G$.
\end{lem}
\section{Inverse Perron values of uniform hypergraphs}
In the following theorem, the relationship between inverse Perron values and connectivity of a hypergraph is presented.
\begin{thm}\rm\label{connected.max}
Let $\mathcal{G}$ be a $k$-uniform hypergraph. Then the following are equivalent:\\
$\rm (1)$ $\mathcal{G}$ is connected.\\
$\rm (2)$ $\alpha_j(\mathcal{G})>0$ for all $j\in V(\mathcal{G})$.\\
$\rm (3)$ $\alpha_j(\mathcal{G})>0$ for some $j\in V(\mathcal{G})$.

\end{thm}
\begin{proof}
(1)$\Longrightarrow$(2). If $\mathcal{G}$ is connected, then by Lemma \ref{connected}, we know that $\alpha_j(\mathcal{G})>0$ for all $j\in V(\mathcal{G})$.

(2)$\Longrightarrow$(3). Obviously.

(3)$\Longrightarrow$(1). Suppose that $\mathcal{G}$ is disconnected. For any $j\in V(\mathcal{G})$, let $\mathcal{G}_1$ be the component of $\mathcal{G}$ such that $j\notin V(\mathcal{G}_1)$. Let $\textbf{x}=(x_1,\ldots,x_{|V(\mathcal{G})|})^\mathrm{T}$ be the vector satisfying
\[{x_i} = \left\{ \begin{gathered}
{{|V({\mathcal{G}_1})|}^{-\frac{1}{k}}},{\kern 1pt} {\kern 1pt} {\kern 1pt} {\kern 1pt} {\kern 1pt} {\kern 1pt} {\kern 1pt} {\kern 1pt} {\kern 1pt} {\kern 1pt} {\kern 1pt} {\kern 1pt} {\kern 1pt} {\kern 1pt} {\kern 1pt} {\kern 1pt} {\kern 1pt} {\kern 1pt} {\kern 1pt} {\kern 1pt} if{\kern 1pt} {\kern 1pt} {\kern 1pt} {\kern 1pt} i \in V({\mathcal{G}_1}), \hfill \\
  0,{\kern 1pt} {\kern 1pt} {\kern 1pt} {\kern 1pt} {\kern 1pt} {\kern 1pt} {\kern 1pt} {\kern 1pt} {\kern 1pt} {\kern 1pt} {\kern 1pt} {\kern 1pt} {\kern 1pt} {\kern 1pt} {\kern 1pt} {\kern 1pt} {\kern 1pt} {\kern 1pt} {\kern 1pt} {\kern 1pt} {\kern 1pt} {\kern 1pt} {\kern 1pt} {\kern 1pt} {\kern 1pt} {\kern 1pt} {\kern 1pt} {\kern 1pt} {\kern 1pt} {\kern 1pt} {\kern 1pt} {\kern 1pt} {\kern 1pt} {\kern 1pt} {\kern 1pt} {\kern 1pt} {\kern 1pt} {\kern 1pt} {\kern 1pt} {\kern 1pt} {\kern 1pt} {\kern 1pt} {\kern 1pt} {\kern 1pt} {\kern 1pt} {\kern 1pt} {\kern 1pt} {\kern 1pt} {\kern 1pt} {\kern 1pt} {\kern 1pt} {\kern 1pt} ~~~~ otherwise. \hfill \\
\end{gathered}  \right.\]
Clearly, we have $\sum\limits_{i=1}^n x_i^k=1$. Then we have $0\leq\alpha_j (\mathcal{G})\leq\mathcal{L}_\mathcal{G}\mathbf{x}^k=0$ for any $j\in V(\mathcal{G})$, a contradiction to (3). Hence $\mathcal{G}$ is connected if (3) holds.
\end{proof}

The \textit{bipartition width} of a hypergraph $\mathcal{G}$ is defined as \cite{bipartition2, bipartition1}
\[
{\rm bw}
(\mathcal{G})=\min \left\{|E(S,\overline{S})|: S\subseteq V(\mathcal{G}),|S|=\left\lfloor\frac{n}{2}\right\rfloor\right\},
\]
where $\left\lfloor\frac{n}{2}\right\rfloor$ denotes the maximum integer not larger than $\frac{n}{2}$. The computation of bw$(\mathcal{G})$ is very difficult even for the graph case. In \cite{Mohar}, Mohar and Poljak use the algebraic connectivity to obtain a lower bound on the bipartition width of a graph. In the following, we use the inverse Perron values to obtain a lower bound on the bipartition width of a uniform hypergraph.
\begin{thm}\label{bwthm}\rm
Let $\mathcal{G}$ be a $k$-uniform  hypergraph with $n$ vertices. Then
\[
{\rm bw(\mathcal{G})} \geq\frac{n+(-1)^n}{k(n+1)}\sum\limits_{j = 1}^n {{\alpha _j}} (\mathcal{G}).
\]
\end{thm}
\begin{proof}
Suppose that $S_0\subseteq V(\mathcal{G})$ satisfying $|S_0|=\left\lfloor\frac{n}{2}\right\rfloor$ and $|E(S_0,\overline{S_0})|={\rm bw}(\mathcal{G})$.
 Let $\textbf{x}=(x_1,\ldots,x_{n})^\mathrm{T}$ be the vector satisfying
\[{x_i} = \left\{ \begin{gathered}
  |S_0{|^{ - \frac{1}
{k}}},~~~~~i \in S_0, \hfill \\
  0,~~~~~~~~~~~~i \in \overline{S_0}. \hfill \\
\end{gathered}  \right.\]
Then $\sum\limits_{i=1}^n x_i^k=1$. For $j\in \overline{S_0}$, we can obtain
\begin{align*}
\alpha_j(\mathcal{G})&\leq \mathcal{L}_{\mathcal{G}}\textbf{x}^k=\sum\limits_{\{i_1,\ldots,i_k\}\in E(\mathcal{G})}\left(x_{i_1}^k+\cdots +x_{i_k}^k-kx_{i_1}\cdots x_{i_k}\right)
\end{align*}
\begin{align*}\label{bw3.1}
\alpha_j(\mathcal{G})&\leq\sum\limits_{\{i_1,\ldots,i_k\}\in E(S_0,\overline{S_0})}\left(x_{i_1}^k+\cdots + x_{i_k}^k-kx_{i_1}\cdots x_{i_k}\right)\\
&=\frac{1}{|S_0|}{\sum\limits_{e\in E(S_0,\overline{S_0})}| e\cap S_0|}
=\frac{t(S_0)\rm bw(\mathcal{G})}{|S_0|},\tag{3.1}
\end{align*}
where $t(S_0)=\frac{1}{\left|E(S_0,\overline{S_0})\right|}{\sum\limits_{e\in E(S_0,\overline{S_0})}| e\cap S_0|}$.

Similarly, for $j\in {S_0}$, we can obtain
\begin{align*}\label{bw3.2}
\alpha_j(\mathcal{G})\leq 
\frac{(k-t(S_0))\rm bw(\mathcal{G})}{|\overline{S_0}|}.\tag{3.2}
\end{align*}
Combining (\ref{bw3.1}) and (\ref{bw3.2}), we can get
\[
\sum\limits_{j=1}^n \alpha_j(\mathcal{G})= \sum\limits_{j\in {S_0}} \alpha_j(\mathcal{G})+ \sum\limits_{j\in \overline{S_0}} \alpha_j(\mathcal{G})
\leq \frac{|S_0|(k-t(S_0))\rm bw(\mathcal{G})}{|\overline{S_0}|}+ \frac{|\overline{S_0}|t(S_0)\rm bw(\mathcal{G})}{|S_0|}.
\]
If $n$ is even, then $|S_0|=|\overline{S_0}|$ and ${\rm bw(\mathcal{G})}\geq\frac{1}{k}\sum\limits_{j=1}^n \alpha_j(\mathcal{G})$. If $n$ is odd, then $|S_0|=|\overline{S_0}|-1=\frac{n-1}{2}$ and
\[
\sum\limits_{j=1}^n \alpha_j(\mathcal{G})\leq k\frac{|\overline {S_0}|}{|{S_0}|}{\rm bw}(\mathcal{G})
=\frac{k(n+1)\rm bw(\mathcal{G})}{n-1},~{\rm bw(\mathcal{G})}\geq\frac{n-1}{k(n+1)}\sum\limits_{j=1}^n \alpha_j(\mathcal{G}).
\]
\end{proof}
The \textit{isoperimetric number} of a $k$-uniform hypergraph $\mathcal{G}$ is defined as
\[
i(\mathcal{G})=\min\left\{\frac{|E(S,\overline{S})|}{|S|}: S\subseteq V(\mathcal{G}), 0\leq S\leq \frac{|V(\mathcal{G})|}{2}\right\}.
\]
Let $\beta(\mathcal{G})= \max\limits_{j\in V(\mathcal{G})}\alpha_j(\mathcal{G})$ denote the maximum inverse Perron value of $\mathcal{G}$. In \cite{liwei}, it is shown that $i(\mathcal{G})\geq \frac{2}{k}\alpha(\mathcal{G})$. We improve it as follows.
\begin{thm}\label{igthm3.4}\rm
Let $\mathcal{G}$ be a $k$-uniform hypergraph. Then
\[
i(\mathcal{G})\geq\frac{\alpha(\mathcal{G})+\beta(\mathcal{G})}{k}.
\]
\end{thm}
\begin{proof}
Suppose that $S_1\subseteq V(\mathcal{G})$ satisfying  $0\leq S_1\leq \frac{|V(\mathcal{G})|}{2}$ and $\frac{|E(S_1,\overline{S_1})|}{|S_1|}=i(\mathcal{G})$.
Let $\textbf{x}=(x_1,\ldots,x_{n})^\mathrm{T}$ be the vector satisfying
\[{x_i} = \left\{ \begin{gathered}
  |S_1{|^{ - \frac{1}
{k}}},~~~~~i \in S_1, \hfill \\
  0,~~~~~~~~~~~i \in \overline{S_1}. \hfill \\
\end{gathered}  \right.\]
Then $\sum\limits_{i=1}^n x_i^k=1$. For $j\in \overline{S_1}$, we can obtain
\begin{align*}\label{ig3.3}
\alpha_j(\mathcal{G})\leq \mathcal{L}_{\mathcal{G}}\textbf{x}^k
=\frac{t(S_1){|E(S_1,\overline{S_1})|}}{|S_1|}=t(S_1)i(\mathcal{G}),\tag{3.3}
\end{align*}
where $t(S_1)=\frac{1}{\left|E(S_1,\overline{S_1})\right|}{\sum\limits_{e\in E(S_1,\overline{S_1})}| e\cap S_1|}$.

Similarly, for   $j\in {S_1}$, we can get
\begin{align*}\label{ig3.4}
\alpha_j(\mathcal{G})\leq
\frac{(k-t(S_1))|E(S_1,\overline{S_1})|}{|\overline{S_1}|}\leq (k-t(S_1))i(\mathcal{G}).\tag{3.4}
\end{align*}
From (\ref{ig3.3}) and (\ref{ig3.4}), we have
\[
\alpha(\mathcal{G})+\beta(\mathcal{G})\leq t(S_1)i(\mathcal{G})+(k-t(S_1))i(\mathcal{G})= k i(\mathcal{G}),~i(\mathcal{G})\geq\frac{\alpha(\mathcal{G})+\beta(\mathcal{G})}{k}.
\]
\end{proof}
The distance $d(u,v)$ between two distinct vertices $u$ and $v$ of $G$ is the length of the shortest path connecting them. The eccentricity of a vertex $v$ is ${\rm ecc}(v)=\max\{d(u,v):u\in V(\mathcal{G})\}$. The diameter and radius of $\mathcal{G}$ are defined as diam$(\mathcal{G})=\max\limits_{v\in V(\mathcal{G})}{\rm ecc}(v)$ and ${\rm rad}(\mathcal{G})=\min\limits_{v\in V(\mathcal{G})}{\rm ecc}(v)$, respectively.
\begin{thm}\rm\label{thm3.4}
Let $\mathcal{G}$ be a connected $k$-uniform  hypergraph with $n$ vertices. Then
\[
{\rm ecc}(j)\geq \frac{k}{2(k-1)(n-1)\alpha_j(\mathcal{G})}, ~j\in V(\mathcal{G}).
\]
\end{thm}
\begin{proof}
For $j\in V(\mathcal{G})$, let $\mathbf{x}=(x_1,\ldots,x_n)^\mathrm{T}\in\mathbb{R}_+^{n}$ satisfying $x_j=0$, $\sum_{i=1}^nx_i^k=1$ and $\alpha_j(\mathcal{G})={\mathcal{L_G}}\mathbf{x}^k$. Then
\begin{align*}\label{shizi3.1}
\alpha_j(\mathcal{G})={\mathcal{L_G}}\mathbf{x}^k
=\sum\limits_{\{i_1,\ldots,i_k\}\in E(\mathcal{G})}\left(x_{i_1}^k+\cdots+x_{i_k}^k- kx_{i_1}\cdots x_{i_k}\right).\tag{3.5}
\end{align*}
From AM-GM inequality, it yields that
\begin{align*}\label{shizi3.2}
\sum\limits_{1\leq s<t\leq k} {x_{i_s}^{\frac{k}{2}} x_{i_{t}}^\frac{k}{2}}
\geq \frac{k(k-1)}{2}\left( \prod\limits_{1\leq s<t\leq k} {x_{i_s}^{\frac{k}{2}} x_{i_{t}}^\frac{k}{2}}\right)^{\frac{2}{k(k-1)}}
=\frac{k(k-1)}{2} x_{i_1}\cdots x_{i_k}.\tag{3.6}
\end{align*}
By (\ref{shizi3.1}) and (\ref{shizi3.2}), we have
\begin{align*}\label{shizi3.3}
&~\alpha_j(\mathcal{G})\geq\sum\limits_{\{i_1,\ldots,i_k\}\in E(\mathcal{G})}\left(x_{i_1}^k+\cdots+x_{i_k}^k- \frac{2}{k-1} \sum\limits_{1\leq s<t\leq k} {x_{i_s}^{\frac{k}{2}} x_{i_{t}}^\frac{k}{2}}\right)\\
&=\frac{1}{k-1}\sum\limits_{\{i_1,\ldots,i_k\}\in E(\mathcal{G})}\sum\limits_{1\leq s<t\leq k}
\left({x_{i_s}^{\frac{k}{2}}- x_{i_{t}}^\frac{k}{2}}\right)^2=\frac{1}{k-1}\sum\limits_{s,t\in e\in E(\mathcal{G})}\left({x_{s}^{\frac{k}{2}}- x_{t}^\frac{k}{2}}\right)^2.\tag{3.7}
\end{align*}
Let $\mathcal{P}=v_{0}e_{1}v_{1}e_{2} \cdots v_{l-1}e_{l}v_{l}$ be the shortest path from vertex $ v_{0}$ to vertex $v_{l}$, where $x_{v_0}=\mathop{\max}\limits_{i\in V(\mathcal{G})}\{x_i\}$, $v_l=j$, $x_{v_l}=0$. Then
\begin{eqnarray*}
&~&\sum_{s,t\in e\in E(\mathcal{G})}\left(x_s^{\frac{k}{2}}-x_t^{\frac{k}{2}}\right)^2
\geq\sum_{s,t\in e\in E(\mathcal{P})}\left(x_s^{\frac{k}{2}}-x_t^{\frac{k}{2}}\right)^2\\
&\geq&\sum_{i=1}^{l}\left(\left(x_{v_{i-1}}^{\frac{k}{2}}-x_{v_{i}}^{\frac{k}{2}}\right)^2+\sum_{u_j\in e_i\backslash\{v_{i-1},v_i\}}\left(\left(x_{v_{i-1}}^{\frac{k}{2}}-x_{u_{j}}^{\frac{k}{2}}\right)^2+
\left(x_{u_{j}}^{ \frac{k}{2}}-x_{v_{i}}^{\frac{k}{2}}\right)^2\right)\right)\\
&\geq& \sum_{i=1}^{l}\left(\left(x_{v_{i-1}}^{\frac{k}{2}}-x_{v_{i}}^{\frac{k}{2}}\right)^2+\frac{1}{2}~\sum_{u_j\in e_i\backslash\{v_{i-1},v_i\}}\left(x_{v_{i-1}}^{\frac{k}{2}}-x_{u_{j}}^{\frac{k}{2}}+x_{u_{j}}^{\frac{k}{2}}-
x_{v_{i}}^{\frac{k}{2}}\right)^2\right)\\
&=&\sum_{i=1}^{l}\left(\left(x_{v_{i-1}}^{\frac{k}{2}}-x_{v_{i}}^{\frac{k}{2}}\right)^2+\frac{k-2}{2} \left(x_{v_{i-1}}^{\frac{k}{2}}-x_{v_{i}}^{\frac{k}{2}}\right)^2\right)\\
&=&\frac{k}{2}\sum_{i=1}^{l}\left(x_{v_{i-1}}^{\frac{k}{2}}-x_{v_{i}}^{\frac{k}{2}}\right)^2.
\end{eqnarray*}
By Cauchy-Schwarz inequality, we obtain
\begin{align*}\label{3.3}
&~\sum_{s,t\in e\in E(\mathcal{G})}\left(x_s^{\frac{k}{2}}-x_t^{\frac{k}{2}}\right)^2
\geq\frac{k}{2}\sum_{i=1}^{l}\left(x_{v_{i-1}}^{\frac{k}{2}}-x_{v_{i}}^{\frac{k}{2}}\right)^2
\geq\frac{k}{2l}\left(\sum_{i=1}^{l}\left(x_{v_{i-1}}^{\frac{k}{2}}-x_{v_{i}}^{\frac{k}{2}} \right)\right)^2\nonumber\\
&=\frac{k}{2l}\left(x_{v_0}^{\frac{k}{2}}-x_{v_l}^{\frac{k}{2}}\right)^2
\geq\frac{k}{2{\rm ecc}(j)}\left(x_{v_0}^{\frac{k}{2}}-x_{v_l}^{\frac{k}{2}}\right)^2\nonumber\geq\frac{k}{2(n-1){\rm ecc}(j)}.\tag{3.8}
\end{align*}
From (\ref{shizi3.3}) and (\ref{3.3}), it yields that
\[\alpha_j(\mathcal{G})\geq \frac{k}{2(k-1)(n-1){\rm ecc}(j)},~{\rm ecc}(j)\geq \frac{k}{2(k-1)(n-1)\alpha_j(\mathcal{G})}.\]
\end{proof}
For a connected $k$-uniform hypergraph $\mathcal{G}$ with $n$ vertices, \cite{liwei} showed that
\begin{align*}\label{3.6}
{\rm diam}(\mathcal{G})\geq\frac{4}{n^2 (k-1)\alpha(\mathcal{G})}.
\end{align*}
By Theorem \ref{thm3.4}, we obtain the following improved result.
\begin{cor}\rm\label{newcor3.4}
Let $\mathcal{G}$ be a connected $k$-uniform  hypergraph with $n$ vertices. Then
\[
{\rm diam}(\mathcal{G})\geq\frac{k}{2(k-1)(n-1)\alpha(\mathcal{G})},~{\rm rad}(\mathcal{G})\geq\frac{k}{2(k-1)(n-1)\beta(\mathcal{G})}.\]
\end{cor}
In \cite{H+}, it is shown that $\alpha(\mathcal{G})\leq\delta$, where $\delta$ is the minimum degree of $\mathcal{G}$. We improve it as follows.
\begin{thm}\rm\label{newthm3.3}
Let $\mathcal{G}$ be a $k$-uniform  hypergraph with $n$ vertices. Then
\[\alpha_j(\mathcal{G})\leq\frac{(k-1)d_j}{n-1},~j\in V(\mathcal{G}).\]
\end{thm}
\begin{proof}
For $j\in V(\mathcal{G})$, let $\textbf{x}=(x_1,\ldots,x_{n})^\mathrm{T}$ be the vector satisfying
\[{x_i} = \left\{ \begin{gathered}
  {(n - 1)^{ - \frac{1}
{k}}},~~~~~~i \ne j, \hfill \\
  0,~~~~~~~~~~~~~~~~~i = j. \hfill \\
\end{gathered}  \right.
\]
Then $\sum\limits_{i=1}^n x_i^k=1$, 
and we can get
\begin{align*}
\alpha_j(\mathcal{G})&\leq \mathcal{L}_{\mathcal{G}}\textbf{x}^k
=\sum\limits_{\{i_1,\ldots,i_k\}\in E(\mathcal{G})}\left(x_{i_1}^k+\cdots +x_{i_k}^k-kx_{i_1}\cdots x_{i_k}\right)\\
&=\sum\limits_{\{i_1,\ldots,i_k\}\in E_j(\mathcal{G})}\left(x_{i_1}^k+\cdots + x_{i_k}^k\right)
=\frac{(k-1)d_j}{n-1}.
\end{align*}
\end{proof}

By Theorem \ref{newthm3.3}, we obtain the following result.
\begin{cor}\rm
Let $\mathcal{G}$ be a $k$-uniform  hypergraph with $n$ vertices, $m$ edges. Then
\[
\sum \limits_ {j=1}^n\alpha_j(\mathcal{G})\leq\frac{(k-1)km}{n-1},~j\in V(\mathcal{G}).
\]
\end{cor}

A $2$-$(n,b,k,r,\lambda)$ design can be regarded as a $k$-uniform $r$-regular hypergraph $\mathcal{G}$ on $n$ vertices, $b$ edges, and $c(x,y)=|\{e\in E(\mathcal{G}): x,y\in e\}|=\lambda$ for any pair of distinct $x,y\in V(\mathcal{G})$. A $2$-design satisfying $n=b$ is called a symmetric design.
\begin{thm}\rm\label{2design}
Let $\mathcal{G}$ be a connected $k$-uniform hypergraph with $n$ vertices. Then $\mathcal{G}$ is a $2$-design if and only if $\alpha_1(\mathcal{G})=\cdots=\alpha_n(\mathcal{G})=\frac{\Delta(k-1)}{n-1}$, where $\Delta$ is the maximum degree of $\mathcal{G}$.
\end{thm}

\begin{proof}
We first prove the necessity. If $\mathcal{G}$ is a $2$-$(n,b,k,r,\lambda)$ design, then $\lambda(n-1)=r(k-1)$ and $\Delta=r=d_1=\cdots=d_n$. For any $j\in V(\mathcal{G})$, by Theorem \ref{newthm3.3}, we have
\begin{align*}\label{newshizi3.2}
\alpha_j(\mathcal{G})\leq\frac{r(k-1)}{n-1}=\lambda.\tag{3.9}
\end{align*}
Let $\mathbf{x}=(x_1,\ldots,x_n)^\mathrm{T}\in\mathbb{R}_+^{n}$ satisfying $x_j=0$, $\sum_{i=1}^nx_i^k=1$ and $\alpha_j(\mathcal{G})={\mathcal{L_G}}\mathbf{x}^k$. Then we get
\begin{align*}\label{newshizi3.3}
\alpha_j(\mathcal{G})= \mathcal{L}_{\mathcal{G}}\textbf{x}^k\geq\sum\limits_{\{i_1,\ldots,i_k\}\in E_j(\mathcal{G})}\left(x_{i_1}^k+\cdots + x_{i_k}^k-kx_{i_1}\cdots x_{i_k}\right)
=\lambda\sum\limits_{i\neq j}x_i^k=\lambda.\tag{3.10}
\end{align*}
Combining (\ref{newshizi3.2}) and (\ref{newshizi3.3}), we can get
\[
\alpha_1(\mathcal{G})=\cdots=\alpha_n(\mathcal{G})=\lambda=\frac{r(k-1)}{n-1}=\frac{\Delta(k-1)}{n-1}.
\]

Next we prove the sufficiency. If $\alpha_1(\mathcal{G})=\cdots=\alpha_n(\mathcal{G})=\frac{\Delta(k-1)}{n-1}$. From Theorem \ref{newthm3.3}, we obtain $d_1=\cdots=d_n=\Delta$. Let $\mathbf{z}=\left({(n-1)^{-\frac{1}{k}}},\ldots,{(n-1)^{-\frac{1}{k}}} \right)^\mathrm{T}\in\mathbb{R}_+^{n-1}$, $\mathbf{y}=\left(\mathbf{z}^\mathrm{T}, 0\right)^\mathrm{T}\in\mathbb{R}_+^n$. Then
\begin{align*}
\mathcal{L}_{\mathcal{G}}\textbf{y}^k&=\sum\limits_{\{i_1,\ldots,i_k\}\in E(\mathcal{G})}\left(y_{i_1}^k+\cdots +y_{i_k}^k-ky_{i_1}\cdots y_{i_k}\right)=\sum\limits_{\{i_1,\ldots,i_k\}\in E_n(\mathcal{G})}\left(y_{i_1}^k+\cdots+ y_{i_k}^k\right)\\
&=\frac{\Delta(k-1)}{n-1}=\alpha_n(\mathcal{G})=\alpha(\mathcal{G}).
\end{align*}
By Lemma \ref{2.1}, we know that $\alpha(\mathcal{G})=\alpha_n(\mathcal{G})$ is the smallest H-eigenvalue of $\mathcal{L_G}(n)$. Since $\mathcal{L}_{\mathcal{G}}(n)\textbf{z}^k=\mathcal{L}_{\mathcal{G}}\textbf{y}^k=\alpha(\mathcal{G})$, $\textbf{z}$ is an H-eigenvector corresponding to $\alpha(\mathcal{G})$, that is
\[
\alpha(\mathcal{G})\mathbf{z}^{[k-1]}=\mathcal{L_G}(n)\mathbf{z}^{k-1}.
\]
For all $i\in V(\mathcal{G})$, we have
\begin{align*}
\alpha(\mathcal{G})&=\frac{1}{{z_i}^{k-1}}\left(\mathcal{L_G}(n)\mathbf{z}^{k-1}\right)_i=\frac{1}{{z_i}^{k-1}}\sum\limits_{i_2,\ldots,i_k=1}^{n-1}(\mathcal{L_G} (n))_{ii_2\cdots i_k}z_{i_2}\cdots z_{i_k}\\
&=\sum\limits_{i_2,\ldots,i_k=1}^{n-1}(\mathcal{L_G})_{ii_2\cdots i_k}
=c(i,n).
\end{align*}
Then we get
\[
c(1,n)=c(2,n)=\cdots=c(n-1,n)=\alpha(\mathcal{G}).
\]
Similarly, we can obtain
\[
c(i,j)=\alpha(\mathcal{G}), ~i,j \in V(\mathcal{G})~ and ~i\neq j,
\]
which implies that $\mathcal{G}$ is a $2$-design.
\end{proof}
we give an estimation of the edge connectivity of a $2$-design as follows.
\begin{thm}\rm
Let $\mathcal{G}$ be a $2$-$(n,b,k,r,\lambda)$ design. Then
\[
\frac{n\lambda}{k}\leq e(\mathcal{G})\leq\frac{(n-1)\lambda}{k-1}.
\]
Moreover, if $\mathcal{G}$ is a symmetric design, then $e(\mathcal{G})=k=r$.
\end{thm}
\begin{proof}
Since $\mathcal{G}$ is a $2$-$(n,b,k,r,\lambda)$ design, we have $\lambda(n-1)=r(k-1)$. By Theorem \ref{2design}, we have \[\alpha(\mathcal{G})=\frac{r(k-1)}{n-1}=\lambda.\]
It follows from Lemma \ref{edge} that
\begin{align*}\label{newshizi3.4}
\frac{n\lambda}{k}=\frac{n}{k}\alpha(\mathcal{G})\leq e(\mathcal{G})\leq r=\frac{(n-1)\lambda}{k-1}.\tag{3.11}
\end{align*}
Moreover, if $\mathcal{G}$ is a symmetric design, then  $n=b$. Since $nr=bk$, we have $r=k$.  From $\lambda(n-1)=r(k-1)$ and (\ref{newshizi3.4}), we have
\[
\frac{n(k-1)}{n-1}\leq e(\mathcal{G})\leq k.
\]
Since $e(\mathcal{G})$ is a positive integer, we can get $e(\mathcal{G})=k=r$.
\end{proof}

\section{Inverse Perron values and resistance distance of graphs}
For a vertex $i$ of a connected graph $G$, we define its resistance eccentricity as $r_i(G)=\max_{j\in V(G)}r_{ij}$.
\begin{thm}\rm
Let $G$ be a connected graph. For any $i\in V(G)$, we have
\begin{eqnarray*}
r_i(G)\leqslant\frac{1}{\alpha_i(G)}.
\end{eqnarray*}
\end{thm}
\begin{proof}
Without loss of generality, assume that $i$ is the vertex corresponding to the last row of the Laplacian matrix $\mathcal{L}_G$. Since $\alpha_i(G)$ is the minimum eigenvalue of the principal submatrix $\mathcal{L}_G(i)$, $\alpha_i^{-1}(G)$ is the spectral radius of the symmetric nonnegative matrix $\mathcal{L}_G(i)^{-1}$. So $\alpha_i^{-1}(G)\geqslant\max_{j\neq i}(\mathcal{L}_G(i)^{-1})_{jj}$.

By Lemma \ref{2.6}, $N=\begin{pmatrix}\mathcal{L}_G(i)^{-1}&0\\0&0\end{pmatrix}\in\mathbb{R}^{n\times n}$ is a symmetric $\{1\}$-inverse of $\mathcal{L}_G$. From Lemma \ref{2.4}, we get $r_{ij}(G)=(\mathcal{L}_G(i)^{-1})_{jj}$ for any $j\neq i$. Hence
\begin{eqnarray*}
\alpha_i^{-1}(G)&\geqslant&\max_{j\neq i}(\mathcal{L}_G(i)^{-1})_{jj}=r_i(G),\\
r_i(G)&\leqslant&\frac{1}{\alpha_i(G)}.
\end{eqnarray*}
\end{proof}
For a vertex $i$ of a connected graph $G$, its resistance centrality is defined as $Kf_i(G)=\sum_{j\in V(G)}r_{ij}(G)$. It is a centrality index of networks \cite{Bozzo}.
\begin{thm}\rm
Let $G$ be a connected graph with $n$ vertices. For any $i\in V(G)$, we have
\begin{eqnarray*}
nKf_i(G)-Kf(G)\leqslant\frac{n-1}{\alpha_i(G)}.
\end{eqnarray*}
\end{thm}
\begin{proof}
Without loss of generality, assume that $i$ is the vertex corresponding to the last row of the Laplacian matrix $\mathcal{L}_G$. Then $\alpha_i^{-1}(G)$ is the maximum eigenvalue of the symmetric matrix $\mathcal{L}_G(i)^{-1}$. Let $e$ be the all-ones column vector, then
\begin{eqnarray*}
\alpha_i^{-1}(G)\geqslant\frac{e^\top\mathcal{L}_G(i)^{-1}e}{e^\top e}=\frac{e^\top\mathcal{L}_G(i)^{-1}e}{n-1}.
\end{eqnarray*}
By Lemma \ref{2.6}, $N=\begin{pmatrix}\mathcal{L}_G(i)^{-1}&0\\0&0\end{pmatrix}\in\mathbb{R}^{n\times n}$ is a symmetric $\{1\}$-inverse of $\mathcal{L}_G$. By Lemma \ref{2.5}, we have
\begin{eqnarray*}
Kf(G)=n\mbox{\rm tr}(N)-e^\top Ne=n\mbox{\rm tr}(\mathcal{L}_G(i)^{-1})-e^\top\mathcal{L}_G(i)^{-1}e.
\end{eqnarray*}
From Lemma \ref{2.4}, we get $r_{ij}(G)=(\mathcal{L}_G(i)^{-1})_{jj}$ for any $j\neq i$. Hence $\mbox{\rm tr}(\mathcal{L}_G(i)^{-1})=Kf_i(G)$ and
\begin{eqnarray*}
Kf(G)=nKf_i(G)-e^\top\mathcal{L}_G(i)^{-1}e.
\end{eqnarray*}
By $\alpha_i^{-1}(G)\geqslant\frac{e^\top\mathcal{L}_G(i)^{-1}e}{n-1}$ we get
\begin{eqnarray*}
\alpha_i^{-1}(G)\geqslant\frac{e^\top\mathcal{L}_G(i)^{-1}e}{n-1}&=&\frac{nKf_i(G)-Kf(G)}{n-1},\\
nKf_i(G)-Kf(G)&\leqslant&\frac{n-1}{\alpha_i(G)}.
\end{eqnarray*}
\end{proof}

\begin{cor}\rm
Let $G$ be a connected graph with $n$ vertices. Then
\begin{eqnarray*}
Kf(G)\leqslant\frac{n-1}{n}\sum_{i=1}^n\alpha_i^{-1}(G).
\end{eqnarray*}
\end{cor}
\begin{proof}
By Theorem 4.2, we have
\begin{eqnarray*}
\sum_{i=1}^n\frac{n-1}{\alpha_i(G)}&\geqslant&\sum_{i=1}^n(nKf_i(G)-Kf(G))=nKf(G),\\
Kf(G)&\leqslant&\frac{n-1}{n}\sum_{i=1}^n\alpha_i^{-1}(G).
\end{eqnarray*}
\end{proof}

\vspace{3mm}
\noindent
\textbf{Acknowledgements}

\vspace{3mm}
This work is supported by the National Natural Science Foundation of China (No. 11371109 and No. 11601102), the Natural Science Foundation of the Heilongjiang Province (No. QC2014C001) and the Fundamental Research Funds for the Central Universities.

\end{spacing}
\end{document}